\documentclass{amsart}

\usepackage{amssymb}
\usepackage{amsfonts}
\usepackage{mathrsfs}

\input xy
\xyoption{all}
\usepackage{natbib} 

\newtheorem{theorem}{Theorem}[section]

\newtheorem{definition}[theorem]{Definition}
\newtheorem{lemma}[theorem]{Lemma}

\newcommand{\arrow}[3]{\ensuremath{#1\colon\!#2\!\rightarrow\!#3}} %draws an arrow

\newcommand{\clarrow}[2]{\ensuremath{#1\!\rightarrow\!#2}}
\newcommand{\clmarrow}[2]{\ensuremath{#1\!\rightarrowtail\!#2}}
\newcommand{\cprod}[2]{\ensuremath{#1\!\times\! #2}} %draws products closer
\newdir{ >}{{}*!/-10pt/@{>}} %shifts an Xy-pic monic arrow tail one space by using @{ >->}

\mathchardef\mhyphen="2D

\title{Zariski cohomology in second order arithmetic}
\author{Colin McLarty}

\begin{document}
\maketitle

\begin{abstract}The cohomology of coherent sheaves and sheaves of Abelian groups on Noetherian schemes are interpreted in second order arithmetic by means of a finiteness theorem.  This finiteness theorem provably fails for the \'etale topology even on Noetherian schemes.
\end{abstract}

\tableofcontents

\section{Overview}

We formalize the cohomology of coherent sheaves and sheaves of Abelian groups on Noetherian schemes in a weak set theory interpretable in second order arithmetic. In brief, much cohomology uses countable sets.  But this doubly understates the result.  Use of countable sets could reach high consistency strength while we need only the strength of second order arithmetic.  And our theorems are not restricted to the countable case.  They apply for example to polynomial rings over $\mathbb{R,C,Z}_p$, or $\mathbb{Q}_p$ if we assume those exist, albeit our foundation cannot prove they do.

This weak set theory is called ZFG[$0$] and it works just like ZFC with the blatant exception that ZFC proves many sets exist which ZFG[$0$] does not.  When ZFC proves a set exists by using power sets or function sets then ZFG[$0$] will generally not prove it exists without some finiteness condition -- such as the Noetherian condition on finite generation of ideals.  The key result in this paper is that ZFG[$0$] suffices to lift this finite generation property to all sheaves of ideals on Noetherian schemes (not only quasi-coherent sheaves of ideals, where it is obvious).

\section{Commutative algebra in ZFG[$0$]}
The set theory ZF[$0$] is Zermelo Fraenkel without the power set axiom.   It proves all sets $A,B$ have a cartesian product \cprod{A}{B}, every equivalence relation \clmarrow{E}{\cprod{A}{A}} on a set $A$ has a set of equivalence classes $A/E$, each set of sets has a union and by familiar tricks it has a disjoint union.  ZF[$0$] does not prove every set of sets has a product, since that is equivalent to power set.   Crucially, it proves every set $A$ has a set $\mathrm{Fin}(A)$ of all finite subsets~\citep{McLHigherorder}.   So for any set $B$ and any finite set $A$ there is a set of all functions \clarrow{A}{B}, since the graph of such a function is a finite subset of \cprod{A}{B}.

We extend ZF[$0$] to ZFG[$0$] by positing a global well-ordering of sets.  That is a linear order $y\leq_{\gamma}\! z$ on sets, such that every non-empty class has a $\leq_{\gamma}$-minimal element and the replacement axiom scheme includes formulas using this relation.  Constructibility along the lines of \citet[pp.~272ff.]{SimpsonBook}\ interprets ZFG[$0$] in second order arithmetic $Z_2$.

All rings in this paper are commutative with unit.

\begin{theorem}\label{T:coproducts}Every set of modules $\{M_i|i\in I\}$ on any ring $R$ has a coproduct $($direct sum$)$ $\coprod_{i\in I}M_i$. 
\end{theorem} 
\begin{proof}  Think of the product \cprod{R}{\cup_iM_i} as the set of formal products $r\cdot y$ for $y$ in any $M_i$.  Take the quotient of the set of finite subsets of that product, identifying two finite sets if they intuitively have the same sum. The obvious addition and multiplication rules make this quotient $\coprod_{i\in I}M_i$.
\end{proof}

As equivalence relations have quotients so submodules \clmarrow{N}{M} have quotient modules $M/N$\! and $R$-modules $M,N$ have tensor products $M \otimes_R N$\!.  The tensor product of $R$-algebras is an $R$-algebra.  For every ring $R$ the category of $R$-modules is cocomplete: every set-sized diagram of $R$-modules has a colimit.  So ZFG[$0$] proves the category of $R$-modules is an AB5 category \citep{Tohoku} with the crucial exception of local smallness.  And in general it does not prove there is a small generator.   We will see it proves the same for many kinds of sheaves of modules on Noetherian schemes.   Finiteness theorems in ZFG[$0$] will show in certain cases that certain function sets and certain sheaves do provably exist.  From there, the argument of  \citet{Tohoku} shows the corresponding sheaf categories have enough injectives and have derived functor cohomology.

The Noetherian condition works as usual:

\begin{theorem}For any ring $R$ the following are equivalent:
\begin{enumerate}
      \item Every ideal of $R$ is finitely generated.
      \item Every increasing chain of ideals is finite.
     \item Every set of ideals of $R$ has maximal elements.
 \end{enumerate}
\end{theorem}
\begin{proof} Familiar proofs of $1\Rightarrow 2\Rightarrow 3\Rightarrow 1$ use no power sets.  
\end{proof}

The finiteness condition implies an existence theorem in ZFG[$0$]:

\begin{theorem}\label{T:ideals}Every Noetherian ring $R$ has a set $\mathrm{Id}(R)$ of all ideals, and a set $\mathrm{Spec}(R)$ of all prime ideals.
\end{theorem}
\begin{proof}Every finite subset of $R$ determines an ideal so by replacement there is a set of all finitely generated ideals.  For  Noetherian $R$ these are all the ideals.
\end{proof}

The usual criteria work since the usual proofs use no power sets:
\begin{theorem}\
\begin{enumerate}
    \item $\mathbb{Z}$ is Noetherian, as is every field.
    \item Every finitely generated algebra over a Noetherian ring is Noetherian.
    \item Every localization of a Noetherian ring is  Noetherian.
\end{enumerate}
\end{theorem}

\subsection{Independence results}\label{S:independence}
Scheme theory in ZFG[$0$] is largely limited to the Noetherian case because ZFG[$0$] does not prove every ring $R$ has a set of all ideals or of all prime ideals.  And even for Noetherian $R$ it does not prove every set of $R$-modules has a direct product, so sheaf constructions need to use finite covers (cf. Theorem~\ref{T:basis}).  These are because ZFG[$0$] has an inner model of countable sets.  The ring of integer polynomials in countably many variables $x_i,i\in \mathbb{N}$ exists by Theorem~\ref{T:coproducts} yet has uncountably many prime ideals.  Every infinite product of non-zero modules is uncountable so it does not exist in the inner model.

\section{The Baer construction for Noetherian rings}\label{S:Baer}

We show in ZFG[$0$] every module on a ring $R$ embeds in an injective $R$-module:

\begin{lemma}A module $M$ on any ring $R$ is injective if every $R$-linear map from an ideal \clarrow{I}{M} extends to an $R$-linear map from all of $R$ to $M$\!.
\end{lemma} 

\begin{proof} Essentially the proof from \cite{BaerAb}, using global choice to avoid needing a set of all partial functions from one module $N$ to another $M$\!.
\end{proof}

If $R$ is Noetherian an $R$-linear map \arrow{i}{I}{M} from any ideal $I$ to any module $M$ is determined by its values on some finite set of generators of $I$\!.  That is, the $R$-linear maps are determined by suitable finite subsets of the product \cprod{I}{M}\!.  By replacement, then, for each ideal $I$ and module $M$ there is a set $(I,M)$ of all $R$-linear maps from $I$ to $M$\!.  And for each $M$ there is a set of all $(I,M)$ as $I$ varies over ideals of $R$. 

Form the coproduct of one copy of $M$ plus a copy $R_i$ of $R$ for each ideal $I$ and map $i\in (I,M)$.  Take the quotient of this product identifying the ideal $I$ in any factor $R$ with its image in $M$ by the corresponding map:
  \[  M_1 =  \big(\cprod{ M}{\coprod_{ (I,M)}     R_i}\,\big)\big/\big(\{\langle m,x_i\rangle| i(x)=m\}\big) \]
Take the obvious inclusion of the first factor \clmarrow{M}{M_1}.  

In general $M_1$ is not injective.  Rather, for any ideal $I$ every $R$-linear map \clarrow{I}{M} extends to an $R$-linear map \clarrow{R}{M_1}.  Baer (amended by \citet[p.~10]{CartEil}) iterated this procedure out to $M_{\alpha}$\! for $\alpha$ the successor cardinal of $\mathrm{Card}(R)$.  Cofinality shows every map from an ideal to $M_{\alpha}$ is actually into some smaller $M_{\beta}$ and so $M_{\alpha}$ is injective.   For $R$ Noetherian we merely need $\alpha$ greater than any finite number.  I.e.~$\alpha=\omega$ suffices. 

\begin{theorem}There is a chain of $M_k$\! for all $k\in \omega$.  Its union is injective and embeds $M$\!.
     \[ \clmarrow{M}{M_{\omega}} = \bigcup_{i\in \omega} M_i \]
\end{theorem}
\begin{proof}By induction each term of the chain exists with a definable embedding into the next.  By replacement the chain forms a set.  The chain union is a quotient of the sum set.  Each ideal $I$ is finitely generated so any $R$-linear map \clmarrow{I}{M} maps a set of generators into some $M_k$ and so maps all of $I$ into that.  By construction that map extends to one from all of $R$ to $M_{k+1}$\! and so to the union.
\end{proof}

So each module $M$ embeds in an injective \clmarrow{M}{M_{\omega}}.  For our purpose rename $M_{\omega}=I_1$.  The quotient module $I_1/M$ embeds in an injective $I_2$ and so on.  Replacement gives an infinite injective resolution~\citep[Appendix~3]{EisComm}:
  \[  \xymatrix{ 0 \ar[r] & M \ar[r] & I_0 \ar[r] & I_2 \ar[r] & \dots \ar[r] & I_i \ar[r] & \dots &  i\in \mathbb{N} } \]
The standard proofs work in ZFG[$0$] to show each definable left exact functorial operation $F$ from $R$-modules to Abelian groups has a definable \emph{right derived functor}.  That means each short exact sequence of $R$-modules  

   \[   \xymatrix{ 0 \ar[r] & M' \ar[r] & M \ar[r] &M'' \ar[r] &  0 } \]
yields a long exact sequence of Abelian groups with the usual naturality. 
\begin{multline*}  \hspace*{5ex} 0 \rightarrow  M' \rightarrow  M \rightarrow M'' \rightarrow  
             R^1F(M')   \rightarrow  R^1F(M) \rightarrow 
     \\  R^1 F (M'') \rightarrow  R^2 F(M')  \rightarrow  R^2  F(M) \rightarrow \dots \hspace*{5ex}
\end{multline*}

This gives all the general theorems of derived functor cohomology for modules over Noetherian rings.  Textbooks, though, generally ignore the foundational issue.  In ZFC and ZFG[$0$] alike the category of modules over a ring is a proper class and not a legitimate entity.   These foundations quantify over modules and sequences of modules; and treat functors only as a shorthand for definable functorial operations.

\section{Noetherian schemes and cohomologies}\label{S:noetschemes}
In the absence of power sets we stipulate that a topological space is a set with a \emph{set} of open subsets meeting the familiar conditions.  So a sheaf is necessarily a set:  if each open subset $U$ of a given space has an associated value $\mathscr{F}(U)$ then by replacement there is a set of all those values.  And so for any point $x\in S$ of a topological space, and sheaf $\mathscr{F}$ on $S$\!, there is a stalk $\mathscr{F}_x$.  That is the colimit of all values $\mathscr{F}(U)$ with $x\in U$\!.

As to schemes, ZFG[$0$] proves existence of any scheme patched together from finitely many spectra of Noetherian rings, but the force of this theorem depends on which Noetherian rings exist (i.e.~provably exist or are assumed to exist for any given purpose).  It does not prove existence of the uncountable $\mathbb{R,C,Z}_p,\mathbb{Q}_p$ though it proves conditional theorems taking their existence as hypotheses.  It does prove existence of all rings finitely generated over the integers $\mathbb{Z}$, rationals $\mathbb{Q}$, finite fields $\mathbb{F}_{p^n}$\!, and their algebraic closures $\overline{\mathbb{Q}},\, \overline{\mathbb{F}}_p$.

More fully the \emph{spectrum} of any ring $R$ is a topological space with the set $\mathrm{Spec}(R)$ of prime ideals of $R$ as set of points, and a closed subset $V(I)$ for each ideal $I$ of $R$.  Intuitively $V(I)$ is the set of points defined by equating each element of $I$ to $0$, and formally it is the set of all prime ideals which contain $I$\!.  Open subsets are the complements of closed subsets.

So every open subset of $\mathrm{Spec}(R)$ is a union of \emph{distinguished} opens $D(f)$ for $f\in R$, where $D(f)$ is the set of all prime ideals not containing $f$\!. The \emph{coordinate ring} on any closed subset $V(I)$ is the quotient ring $R/I$\!, and on any distinguished open $D(f)$ it  is the localized ring $R_f$\!.  The coordinate ring $\mathcal{O}(U)$ on an arbitrary open subset $U\subseteq \mathrm{Spec}(R)$ is a subring of the product of the coordinate rings on any cover of $U$\! by distinguished opens~\citep[p.~76]{HartAG}.

In general ZFG[$0$] does not prove these things exist.  But for Noetherian $R$ it proves $\mathrm{Spec}(R)$ exists and has a \emph{Noetherian topology}.  Every open subset is compact.  Notably, every open subset is a finite union of distinguished opens so coordinate rings on all opens exist as subrings of finite products of the rings $R_f$\!.

 A central theorem of sheaf theory must be limited to Noetherian spaces here:

\begin{theorem}\label{T:basis}  
For any base $\mathscr{B}$ of the topology of a Noetherian space $X$\!:  Given a set $\mathscr{F}(U)$ for each open subset $U\in \mathscr{B}$, and a restriction  map \clarrow{\mathscr{F}(U)}{\mathscr{F}(U')} for each inclusion $U'\subseteq U$ of basis opens, if the data satisfy the sheaf conditions so far as they are defined then they extend to determine a unique sheaf $\mathscr{F}$\!.
\end{theorem}
\begin{proof}  The usual proof forms products $\prod_i(\mathscr{F}(U_i))$ for basis open covers $\{U_i\}$ of arbitrary open subsets.  For Noetherian spaces, use finite covers.
\end{proof}

The yields the standard theory of Noetherian schemes and scheme morphisms and of all sheaves of modules on those schemes.  Of course the issue remains of proving which sheaves of modules exist.

A key fact is that the sheaf associated to a presheaf on a Noetherian space can be constructed as a colimit of finite products using finite covers of open subsets \citep[pp.~230ff.]{VerdierTopologies}.  So colimits of sets lift to cocompleteness of the category of  sheaves of sets on any Noetherian space, and colimits of modules lift to cocompleteness of the category of sheaves of modules on any Noetherian scheme.

For a scheme morphism \arrow{f}{X}{Y} the direct image functor $f_*$ from sheaves of sets on $X$ to those on $Y$ uses no power sets or limits.  The inverse image functor $f^{-1}$ from sheaves of sets on $Y$ to those on $X$ uses sheafification of a colimit.  Then $f_*$ lifts directly to a functor from sheaves of modules on $X$ to those on $Y$, while  $f^{-1}$ lifts to an inverse image functor $f^*$ on sheaves of modules by sheafifying tensor products.  The usual treatments work in ZFG[$0$].

The theory of sheaves of modules in \citet[p.~108--29]{HartAG} works in ZFG[$0$] with one exception: Hartshorne~(p.~109) treats the sets of homorphisms $\mathrm{Hom}(\mathscr{F\!,G})$ and homomorphism sheaves $\mathscr{H}\!om(\mathscr{F\!,G})$ as existing for all sheaves of modules on a given scheme.  This is not provable in ZFG[$0$].  But it is for $\mathscr{F}$ coherent on a Noetherian scheme so each homomorphism is determined by its values on finitely many generating sections of $\mathscr{F}$\!.  We must define coherent sheaves.

Every module $M$ on a Noetherian ring $R$ gives a sheaf of modules $\widetilde{M}$ on $\mathrm{Spec}(R)$  whose value over any distinguished open $U_f$ is the localization $M_f$\!.  This is Theorem~\ref{T:basis} plus straightforward calculation.  A sheaf $\mathscr{F}$ on a scheme $X$ is \emph{quasi-coherent} if $X$ is covered by spectra $\mathrm{Spec}(R)$ such that the restriction of $\mathscr{F}$ to each $\mathrm{Spec}(R)$ is (isomorphic to) $\widetilde{M}$ for some $R$-module $M$\!.  It is \emph{coherent} if each of these $M$ is finitely generated over its $R$.  Hartshorne's proofs on quasi-coherent and coherent sheaves work verbatim in ZFG[$0$] as they consistently uses compactness to reduce questions to finite covers and finite products.

\citet{HartAG} relates three cohomology theories for schemes.  He defines \emph{cohomology} by derived functors on the category of Abelian groups on a topological space~(p.~207).  He quickly relates this to derived functors on the category of modules on the structure sheaf of any scheme~(p.~208).  Third is \v Cech cohomology~(pp.~218ff.) defined by quotients of finite products, so it naturally works in ZFG[$0$].  The problem is to prove the derived functors are well defined, i.e.~to prove in ZFG[$0$] the relevant categories have enough injectives.  Our Section~\ref{S:sheafinj} does this.  

Crucial theorems relating these cohomologies involve infinite \emph{direct limits} but no infinite \emph{inverse limits} in Hartshorne's terminology.  In our terminology they use infiite colimits but no infinite limits.  So they work in ZFG[$0$].  

The chief use of $\mathrm{Hom}(\mathscr{F\!,G})$ and $\mathscr{H}\!om(\mathscr{F\!,G})$ in Hartshorne is Serre duality (pp.~239ff.), a result on coherent sheaves on projective varieties over fields.  The usual proofs naturally work in ZFG[$0$].

\section{Injectives for Zariski sheaves}\label{S:sheafinj}
\citet[p.~207]{HartAG} uses Godement's construction of injective embeddings, which fails in ZFG[$0$] as it uses infinite products.   Rather, we extend Section~\ref{S:Baer}.  

\subsection{Finite generation of sheaves of ideals}
The chief issue is to show in ZFG[$0$] that all ideals of structure sheaves of Noetherian schemes are finitely generated (not only quasi-coherent sheaves of ideals), and the same for all ideals of the constant sheaf of integers on any Noetherian space. 

\subsubsection{Sheaves of modules}\label{S:digraphs}
The problem reduces to the spectra $\mathrm{Spec}(R)$ of Noetherian rings.  For any sheaf of ideals $\mathscr{I}$\! of the structure sheaf $\mathcal{O}_R$ write $I$ for the ideal of global sections.   If $\mathscr{I}$\! is quasi-coherent then $\mathscr{I}\!(D(f))$ is the localization $I_f$ for all distinguished opens $D(f)$, but in any case $I_f\subseteq \mathscr{I}\!(D(f))$.

\begin{definition} For any ring $R$ a \emph{digraph of ideals} is a rooted directed graph with nodes $\langle D(f),K\rangle$ with $K$ an ideal of the localization $R_f$\!.  It must be 
\begin{itemize}
  \item Global: the root is $\langle\mathrm{Spec}(R),I\rangle$ for some ideal $I$\!.
  \item Functional: any open $D(f)$ occurs in at most one node.
  \item Decreasing on opens: an edge \clarrow{\langle D(g),H\rangle}{\langle D(f),K\rangle} implies  
              $D(f)\subsetneq D(g)$.
  \item Increasing on ideals:  an edge \clarrow{\langle D(g),H\rangle}{\langle D(f),K\rangle} implies $K$ properly includes the localization $H_f$ of $H$ to $D(f)$.
\end{itemize}

\end{definition}
\noindent Notice a node $\langle D(f),K\rangle$ is identified by the distinguished open subset $D(f)$ and nothing depends on which function $f$ is chosen to specify it. 

\begin{lemma} Every finite digraph of ideals on any affine scheme $\mathrm{Spec}(R)$ \emph{generates} a sheaf $\mathscr{I}$\! of ideals.
\end{lemma}
\begin{proof} For any distinguished open $D(h)$, each node $\langle D(f),K\rangle$ of the digraph has an associated localization $K_h$ of its ideal to the intersection $D(f)\cap D(h)$; and $D(h)$ is covered by these intersections.  Define the value $\mathscr{I}\!(D(h))$ by patching together all compatible families of elements of these localizations $K_h$ over $D(h)$.  This involves taking the product of all the $K_h$, but there are only finitely many.
\end{proof} 

\begin{definition} For any ring $R$ a \emph{digraph of global generators} is a rooted directed graph with nodes $\langle D(f),G_f\rangle$ with $D(f)$ a distinguished open and $G_f$ a finite subset of $R$, such that the associated nodes $\langle D(f),K\rangle$ form a digraph of ideals where $K$ is the ideal of $R_f$ generated by $G_f$.  Note $G_f$ is a subset of $R$, not only of $R_f$.
\end{definition}

\begin{lemma} Every digraph of ideals has a digraph of global generators on the same open subsets $D(f)$.
\end{lemma}
\begin{proof}  Each localization $R_f$ is Noetherian, and clearing denominators in any finite set of generators for an ideal of $R_f$ gives generators all in $R$ for that same ideal.
\end{proof} 
 
The point for us is:

\begin{lemma}\label{L:digraph}Every sheaf of ideals $\mathscr{I}$\! on a Noetherian affine scheme $\mathrm{Spec}(R)$ is generated by a finite digraph of ideals (not unique).
\end{lemma} 
\begin{proof} We produce a digraph by successive generations.  The first generation is the root $\langle\mathrm{Spec}(R),I\rangle$ for $I$ the global sections of $\mathscr{I}$\!.  Call an open $D(h)$ \emph{expansive} from $\mathrm{Spec}(R)$ if the value $\mathscr{I}\!(D(h))$ is strictly larger than the localization $I_h$.  The union of all these opens is the union, generally in many ways,  of some finite set of them.  Choose such a finite set $\{D(h_1),\dots,D(h_n)\}$ and let the children of the root be the pairs $\langle D(h_i),\mathscr{I}\!(D(h_i))\rangle$.  Call this the second generation. 

For the third generation repeat that reasoning with each $h_i$ in place of $\mathrm{Spec}(R)$.  Call an open $D(j)\subseteq D(h_i)$ \emph{expansive} from $D(h_i)$ if the value $\mathscr{I}\!(D(j))$ is strictly larger than the localization $\mathscr{I}\!(D(h_i))_j$.  The union of all these opens is the union, generally in many ways,  of some finite set of them $\{D(j_1),\dots,D(j_m)\}$.  Let the children of the root be the pairs $\langle D(j_i),\mathscr{I}\!(D(j_i))\rangle$.  Continue forming new generations as long as there are expansive opens.

Each node has finitely many children. Each directed path is finite since it gives a strictly increasing chain of ideals of $R$.  By K\H onig's lemma the digraph is finite.  So it generates a subsheaf $\mathscr{I'}$ of $\mathscr{I}$\!.  To show $\mathscr{I'=I}$ it suffices to show each section $s\in \mathscr{I}\!(D(f))$ on a distinguished open is also in $\mathscr{I'}(D(f))$.  

Either $s\in \mathscr{I'}(D(f))$ or $D(f)$ is expansive from the root $\mathrm{Spec}(R)$.    In the latter case, $s$ is covered by its restrictions $s_i\in \mathscr{I}\!(D(h_i)\cap D(f))$ for each child $\langle D(h_i),\mathscr{I}\!(D(h_i))\rangle$ of the root.  Repeat that reasoning for each generation.  A directed path ends only when it has no expansive opens below it, i.e.~when it reaches a section in $\mathscr{I'}$.  And each directed path is finite. So the section $s$ is covered by sections in $\mathscr{I'}$.
\end{proof}

\begin{theorem}\label{T:zariskifiniteideal}Every sheaf of ideals $\mathscr{I}$\! on a Noetherian affine scheme $\mathrm{Spec}(R)$ is generated by a finite digraph of global generators (not unique).
\end{theorem}

\subsubsection{Sheaves of Abelian groups}

For any Noetherian space $X$ notice each open subset $U\subseteq X$\! is covered by connected open subsets.  Without loss of generality we can assume $X$\! is connected so for any constant sheaf on $X$\! the restriction map from $X$\! to any connected open subset is an isomorphism.

For any sheaf $\mathscr{I}$\! of ideals of the constant sheaf $\mathbb{Z}$ on $X$\! and connected open subsets $U\subseteq V\subseteq X$\!, call $U$\! \emph{expansive} from $V$\! if $\mathscr{I}(V)\subsetneq \mathscr{I}(U)$.  A simple analog of Lemma~\ref{L:digraph} shows $\mathscr{I}$\! is generated by a finite digraph of global generators.   

We state the following theorems for sheaves of modules on the structure sheaf on Noetherian schemes.  Their analogs for sheaves of Abelian groups are similar.

\begin{theorem}For any Noetherian scheme $X$ there is a set $\mathcal{I}d(X)$ of all sheaves of ideals of the structure sheaf $\mathcal{O}_X$\!.
\end{theorem}
\begin{proof} The problem reduces to finitely many affine Noetherian schemes $\mathrm{Spec}(R),\mathcal{O}_R$ covering $X$\!.  By replacement it suffices to have a set of all finite digraphs of ideals of $\mathcal{O}_R$.  But replacement gives a set $\mathcal{D}(R)$ of all distinguished opens of $\mathrm{Spec}(R)$.  Each localization $R_f$ has a set of ideals, so replacement and sum set give a set $\mathcal{L}(R)$ of all ideals of localizations.  So each sheaf in $\mathcal{I}d(R)$ is given by some finite subset of the cartesian square $(\cprod{\mathcal{D}(R)}{\mathcal{L}(R)})^2$.
\end{proof}

\begin{theorem}\label{T:abgrpfiniteideal} For any Noetherian scheme $X$\! and any sheaf of ideals $\mathscr{I}$\! and sheaf of modules $\mathscr{M}$\! on $X$ there is a set $(\mathscr{I\!,M})$ of all $\mathcal{O}_X$-module maps \clarrow{\mathscr{I}\!}{\mathscr{M}}\!.  So for fixed $\mathscr{M}$\! there is a set of all  $(\mathscr{I\!,M})$.
\end{theorem}
\begin{proof} The first claim implies the second by replacement.  Again the problem reduces to $\mathrm{Spec}(R)$ in any finite affine cover.  Given $\mathscr{I\!,M}$\! on $\mathrm{Spec}(R)$ consider any finite digraph of generators.  Each map  \clarrow{\mathscr{I}\!}{\mathscr{M}}\! is determined by the finitely many values of $s\in R$ in the digraph.
\end{proof}

\subsection{The T\^{o}hoku construction}
The argument of Section~\ref{S:Baer} adapts directly to sheaves of modules and ideals on a Noetherian scheme, and to sheaves of Abelian groups on any Noetherian space.  For elements of groups, modules, and ideals take sections (partial or global) of the sheaves.   Where Section~\ref{S:Baer} refers to a non-disjoint sum this proof uses a pushout of submodules.  All covers can be assumed finite.  This gives all the general theorems of derived functor cohomology for these cases, understanding that these categories of sheaves do not exist as legitimate entities.   They are shorthand for certain definable classes, and functors between them are shorthand for definable functorial operations.

\section{Counterexample in \'etale  cohomology}

The finiteness assertions analogous to Theorems~\ref{T:zariskifiniteideal} and \ref{T:abgrpfiniteideal} fail for \'etale cohomology.  A single ideal of the \'etale structure sheaf of a Noetherian scheme can hold information about arbitrarily high degree covers of that scheme and so not be finitely generated. Our counterexample occurs over a perfectly natural base namely the punctured line over any algebraically closed field $k$ of characteristic $\neq 2$.   Because it is only a counterexample, we will take more basic facts about \'etale sheaves for granted than we did for the positive results on Zariski sheaves above.  See  \citet[Thm.~03OJ]{StacksAlgebra} for the definition of \'etale structure sheaves.

Consider the series of  iterated double covers of the punctured $k$ line:

\[ \mathrm{Spec}(k[x_0,x_0^{-1}]) \leftarrow \mathrm{Spec}(k[x_1,x_1^{-1}]) \leftarrow \mathrm{Spec}(k[x_2,x_2^{-1}]) \leftarrow \dots \]
given by 
  \[ x_0\ \mapsto\ x_1^2 \qquad x_1\ \mapsto\ x_2^2\ \qquad \dots \] 

These covers have a lot of symmetries, and symmetries tend to defeat our purpose.  So remove points from each to break the symmetry over the preceding one, but  keep covering the preceding one.   For example, from each $\mathrm{Spec}(k[x_n,x_n^{-1}])$ delete the point $(x_n-2)$, plus deleting all of the points that lie over points deleted from earlier covers.
Then the general term in the series is
 \[ X_n = \mathrm{Spec}(k[x_n,x_n^{-1},((x_n^{2n}-2)\cdot(x_n^{2(n-1)}-2)\cdots (x_n^2-2))^{-1}]) \]
Let $\mathcal{C}$ be the full subcategory of these covers in the \'etale site over the punctured $k$ line.  It is just the sequence
 \[   \mathrm{Spec}(k[x_0,x_0^{-1}])=X_0 \leftarrow  X_1 \leftarrow X_2 \leftarrow \dots \]

For each $X_n$ take the ideal $(x_n-1)$. These ideals form a presheaf ideal $\mathscr{I}$\! of the \'etale structure sheaf restricted to $\mathcal{C}$.  What matters is that each ideal is strictly larger than the presheaf property requires, so the presheaf requires infinitely many generators.  We will see this remains true for the \'etale sheaf of ideals that this presheaf generates. 

The sections of $\mathscr{I}$\! generate a sheaf of ideals of the \'etale structure sheaf $\mathscr{O}_{X_0et}$ on the punctured $k$ line.   In more detail, there is a well defined presheaf of sets assigning to each \'etale open \clarrow{\mathrm{Spec}(S)}{\mathrm{Spec}(k[x_0,x_0^{-1}])} the set of all restrictions to $\mathrm{Spec}(S)$ of sections in  $\mathscr{I}$\!.  Now over each \clarrow{\mathrm{Spec}(S)}{\mathrm{Spec}(k[x_0,x_0^{-1}])} form the set of all $S$-linear combinations of sections, to get a presheaf $\mathscr{I}^*$ of ideals on the \'etale site.  Since $\mathscr{I}^*$\! is a subpresheaf of a sheaf it is separated, so pasting together all compatible families of sections gives an \'etale sheaf $\mathscr{I}_{et}$.

\begin{lemma}\label{L:proper} The \'etale sheaf is a proper ideal:  $\mathscr{I}_{et}\neq \mathscr{O}_{X_0et}$.
\end{lemma}
\begin{proof} Since every \'etale cover has a finite subcover, and every section of $\mathscr{I}_{et}$ is covered by a finite linear combination of restrictions of sections of the original $\mathscr{I}$\!, equality $\mathscr{I}_{et}= \mathscr{O}_{X_0et}$ would imply that in some $X_n$ the unit section on some neighborhood of the point $(x_n-1)\in X_n$ is in the ideal  $(x_n-1)$ which is absurd.
\end{proof}

\begin{lemma} For each of the original $X_i$ we have $\mathscr{I}(X_i)=\mathscr{I}_{et}(X_i)$.
\end{lemma}
\begin{proof} Obviously $\mathscr{I}(X_i)\subseteq \mathscr{I}_{et}(X_i)$.  And $\mathscr{I}(X_i)=(x_i-1)$ is a maximal ideal in its ring.  So if $\mathscr{I}(X_i)=\mathscr{I}_{et}(X_i)$ fails for some $X_i$ then that $\mathscr{I}_{et}(X_i)$ is the unit ideal in its ring.  But each $X_i$ covers the base, so this implies $\mathscr{I}_{et}= \mathscr{O}_{X_0et}$.
\end{proof}

\begin{theorem} No finite set of sections of $\mathscr{I}_{et}$ generates the whole.
\end{theorem}
\begin{proof} Since every \'etale cover has a finite subcover, and every section of $\mathscr{I}_{et}$ is covered by a linear combination of restrictions of sections of the original $\mathscr{I}$\!, if finitely many sections of $\mathscr{I}_{et}$  generate the whole then some finite number of restrictions of sections of the $\mathscr{I}$\!s generate the whole.  Choose $n$ such that all are restrictions of sections of $X_n$.  But then  $\mathscr{I}(X_n)$ would already generate the presheaf  $\mathscr{I}$ on the subcategory $\mathcal{C}$ which is absurd.
\end{proof}
 
%\bibliographystyle{apalike}
%\bibliography{cohomologyrefs,MacLane}

\end{document}